\documentclass[10pt]{article}

\usepackage{amstext}
\usepackage{amssymb}
\usepackage{amsmath}
\usepackage{amsthm}
\usepackage{fullpage}
\usepackage[dvips]{graphicx}        % ¬ Єа®Ї ЄҐв ¤«п ўЄ«озҐ­Ёп аЁбг­Є®ў eps Є®¬ ­¤®© \includegraphicx{XXX.eps}
\graphicspath{{./pict/}}            % Џгвм Є Ї ЇЄҐ б аЁбг­Є ¬Ё

\usepackage[cp1251]{inputenc}
\usepackage[russian,english]{babel}

\theoremstyle{plain}
\newtheorem{theorem}{Theorem}[section]

\newtheorem{prepos}[theorem]{Proposition}

\theoremstyle{definition}
\newtheorem{definition}[theorem]{Definition}

\newtheorem{remark}[theorem]{Remark}

\def\sgn{\mathop{\rm sign}\nolimits}

\def\sgn{\mathop{\rm sign}\nolimits}

\def\sgn{\mathop{\rm sign}\nolimits}

\begin{document}

\title{
On the spectra of Schwarz matrices with certain sign patterns}
% Leave blank; editors will write the exact dates above

%\author{
%Mikhail Tyaglov\thanks{Technische Universit\"at Berlin, Institut f\"ur Mathematik,
%MA 4-5, Strasse des 17. Juni 136, 10623 Berlin, Germany
%(tyaglov@math.tu-berlin.de). Supported by the European Research Council under the
%European Union’s Seventh Framework Programme (FP7/2007-2013)/ ERC grant agreement no.~259173.}
%}

\author{Mikhail Tyaglov\thanks{Supported by the European Research Council under the
European Union’s Seventh Framework Programme (FP7/2007-2013)/ ERC grant agreement no.~259173.}\\
\small Technische Universit\"at  Berlin, Institut f\"ur Mathematik,\\
\small MA 4-5, Strasse des 17. Juni 136, 10623, Berlin, Germany\\
Email: {\tt tyaglov@math.tu-berlin.de}}

\date{\small \today}

%\pagestyle{myheadings}
%\markboth{M.\ Tyaglov}{On the spectra of Schwarz matrices with certain sign patterns}
\maketitle

\vspace{5mm}


\vspace{2mm}


\vspace{5mm}

\begin{abstract}
The direct and inverse spectral problems are solved for a wide subclass of the class of Schwarz matrices. A connection between the Schwarz matrices and the so-called generalized Hurwitz polynomials is found. The known results due to H.~Wall and O.~Holtz are briefly reviewed and obtained as particular cases.
\end{abstract}

%\begin{keywords}
%Schwarz matrices, inverse spectral problem, generalized Hurwitz polynomials, continued fractions, tridiagonal matrices.
%\end{keywords}
%\begin{AMS}
%15A29, 47A75, 47B36, 26C10.
%\end{AMS}

\setcounter{equation}{0}

%%%%%%%%%%%%%%%%%%%%%%%%%%%%%%%%%%%%%%%%%%%%%%%%%%%%%%%%%%%%%%%%%%%%%%%%%%%%%%%
\section{Introduction}\label{section:intro}
%%%%%%%%%%%%%%%%%%%%%%%%%%%%%%%%%%%%%%%%%%%%%%%%%%%%%%%%%%%%%%%%%%%%%%%%%%%%%%%

In this work,  we consider the matrices of the form
\begin{equation}\label{main.matrix.intro}
\left[\begin{array}{cccccc}
    -b_0 & 1 &  0 &\dots&   0   & 0 \\
    -b_1 & 0 &1 &\dots&   0   & 0 \\
     0  &-b_2 & 0 &\dots&   0   & 0 \\
    \vdots&\vdots&\vdots&\ddots&\vdots&\vdots\\
     0  &  0  &  0  &\dots&0& 1\\
     0  &  0  &  0  &\dots&-b_{n-1}&0\\
\end{array}\right],\qquad b_k\in\mathbb{R}\backslash\{0\}
\end{equation}
that usually called the \emph{Schwarz matrices}\footnote{Schwarz himself~\cite{Schwarz} considered also the matrices whose $(1,1)$th entry is zero while $(n,n)$th entry is nonzero. Sometimes such matrices are called the Schwarz matrices as well (see e.g.~\cite{Elsner_Hersh}).}. We solve direct and inverse problems for such matrices with certain sign patterns.

These matrices are well-studied from the matrix theory point of view (see e.g.~\cite{Chen_Chu,Puri_Weygandt,Datta1,Datta2,Datta3,Datta4} and references there). Here we use the method due to Wall~\cite{Wall'45,Wall} to solve the inverse spectral problem and our results on the generalized Hurwitz polynomials to solve the direct spectral problem for the Schwarz matrices~(\ref{main.matrix.intro}) with a wide class of sign patterns. The case of all $b_k$ positive
was considered by H.\,Wall~\cite{Wall'45} and later by H.\,Schwarz~\cite{Schwarz} and many other authors. The case of all $b_k$ negative was considered by O.\,Holtz~\cite{H4}. Here we use formul\ae~obtained by Wall in~\cite{Wall'45} which connect the entries of the matrix~(\ref{main.matrix.intro}) with coefficients of its characteristic polynomial (see formul\ae~(\ref{Formulae.1}) below) to use the so-called generalized Hurwitz theorem established in~\cite{Tyaglov.gen.Hurw}.

In Section~\ref{section:Wall.and.Schwarz}, we review results due to Wall that was obtained in~\cite{Wall'45}. Section~\ref{section:solved} is devoted to all solved direct and inverse problems for the Schwarz matrices. In Section~\ref{section:generalized.Hurwitz}, we recall some basic facts on the generalized Hurwitz polynomials established in~\cite{Tyaglov.gen.Hurw}. In Section~\ref{section:spectral.problems}, we prove our main theorems on the direct and
inverse problems for Schwarz matrices with certain sign patterns. Finally, in Section~\ref{section:examples}, we apply our results of Section~\ref{section:spectral.problems}
to matrices~(\ref{main.matrix.intro}) with one sign change in the sequence $b_1,\ldots,b_{n-1}$.
In particular, we prove the direct and inverse problems for the matrices~(\ref{main.matrix.intro}) with $b_1>0$, $b_2,\ldots,b_{n-1}<0$ which was considered in~\cite{Bebiano'11}.

\setcounter{equation}{0}

%%%%%%%%%%%%%%%%%%%%%%%%%%%%%%%%%%%%%%%%%%%%%%%%%%%%%%%%%%%%%%%%%%%%%%%%%%%%%%%
\section{Wall's continued fractions and the Schwarz matrices}\label{section:Wall.and.Schwarz}
%%%%%%%%%%%%%%%%%%%%%%%%%%%%%%%%%%%%%%%%%%%%%%%%%%%%%%%%%%%%%%%%%%%%%%%%%%%%%%%

Given a monic real polynomial
\begin{equation}\label{main.poly}
p(z)=z^n+a_1z^{n-1}+\cdots+a_n,
\end{equation}
we represent it as follows
$$%$$%\begin{equation*}
p(z)=p_0(z^2)+zp_1(z^2),
$$%\end{equation*}
where the polynomials $p_0(u)$ and $p_1(u)$ are the even and odd parts of the polynomial~$p$, respectively:
\begin{equation}\label{even.part}
p_0(u)=a_n+a_{n-2}u+a_{n-4}u^2+\cdots,
\end{equation}
\begin{equation}\label{odd.part}
p_1(u)=a_{n-1}+a_{n-3}u+a_{n-5}u^2+\cdots
\end{equation}
We also introduce the polynomial $q$ as follows
\begin{equation}\label{poly.q}
q(z)=
\left\{\begin{array}{l}
p_0(z^2)\quad\,\,\,\mathrm{if}\quad n=2l+1,\\
zp_1(z^2)\quad\mathrm{if}\quad n=2l.
\end{array}\right.
\end{equation}

Let us associate with the polynomial $p$ the following determinants called the \textit{Hurwitz determinants}:
\begin{equation}\label{delta}
\Delta_{j}(p)=
\det\left[ \begin{array}{cccccc}
a_1&a_3&a_5&a_7&\dots&a_{2j-1}\\
1&a_2&a_4&a_6&\dots&a_{2j-2}\\
0  &a_1&a_3&a_5&\dots&a_{2j-3}\\
0  &1&a_2&a_4&\dots&a_{2j-4}\\
\vdots&\vdots&\vdots&\vdots&\ddots&\vdots\\
0  &0  &0  &0  &\dots&a_{j}
\end{array}
\right],\quad j=1,\ldots,n,
\end{equation}
where we set $a_i\equiv0$ for $i>n$.

In 1945, H.\,Wall established~\cite{Wall'45} (see also~\cite{Wall}) the following theorem.
\begin{theorem}[Wall]\label{Theorem.Wall}
If the coefficients of the polynomial $p$ given in~(\ref{main.poly}) satisfy the inequalities
\begin{equation}\label{Hurw.det.ineq}
\Delta_j(p)\neq0,\qquad j=1,\ldots,n,
\end{equation}
then there is determined uniquely a continued fraction of the form
\begin{equation}\label{Wall.cont.frac}
\displaystyle\frac{q(z)}{p(z)}=\frac{b_0}{z+b_0+\displaystyle\frac{b_1}{z+\displaystyle\frac{b_2}{\ddots+\displaystyle\frac{b_{n-1}}{z}}}}
%\frac{q(z)}{p(z)}=\frac{b_0}{z+b_0}\underset{+}{}\, \frac{b_1}{z}\underset{+}{}\,
%\frac{b_2}{z}\underset{+}{}\underset{\cdots}{}\underset{+}{}\, \frac{b_{n-1}}{z}
\end{equation}
where $q$ is defined in~(\ref{poly.q}) and the real coefficients $b_k$ are given by the formul\ae
\begin{equation}\label{Formulae.1}
\begin{array}{l}
b_0=\Delta_1(p),\\
 \\
\displaystyle b_k=\frac{\Delta_{k-2}(p)\Delta_{k+1}(p)}{\Delta_{k-1}(p)\Delta_k(p)},\qquad k=1,\ldots,n-1,
\end{array}
\end{equation}
where $\Delta_{-1}(p)=\Delta_0(p)\equiv1$.

Conversely, the coefficients in the last denominator of a continued fraction of the form~(\ref{Wall.cont.frac}) satisfy the inequalities~(\ref{Hurw.det.ineq}).
\end{theorem}

From the form of the continued fraction~(\ref{Wall.cont.frac}) it is easy to see that
$$%\begin{equation*}
\displaystyle \frac{q(z)}{p(z)}=b_0\,((zE_n-J_n)^{-1}e_1,e_1),
$$%\end{equation*}
where $e_1$ is the first coordinate vector in $\mathbb{R}^{n}$, $E_n$ is the $n\times n$ unity matrix, and
\begin{equation}\label{main.matrix}
J_n=
\left[\begin{array}{cccccc}
    -b_0 & 1 &  0 &\dots&   0   & 0 \\
    -b_1 & 0 &1 &\dots&   0   & 0 \\
     0  &-b_2 & 0 &\dots&   0   & 0 \\
    \vdots&\vdots&\vdots&\ddots&\vdots&\vdots\\
     0  &  0  &  0  &\dots&0& 1\\
     0  &  0  &  0  &\dots&-b_{n-1}&0\\
\end{array}\right],
\end{equation}
where the nonzero real entries $b_k$ are exactly the coefficients of the continued fraction~(\ref{Wall.cont.frac}). In other words, the polynomial $p(z)$ is the characteristic polynomial of the matrix $J_n$, while the polynomial $q(z)/b_0$ is the characteristic polynomial
of the principal submatrix of the matrix $J_n$ obtained by deleting the first column and the
first row. Thus, we come to the following conclusion.
\begin{theorem}\label{Theorem.Schwarz.matrix.Wall}
The characteristic polynomial $p$ of the matrix $J_n$ defined in~(\ref{main.matrix}) satisfies
the inequalities~(\ref{Hurw.det.ineq}). Conversely, for every real polynomial $p$ satisfying the inequalities~(\ref{Hurw.det.ineq}), there exists a unique matrix of the form~(\ref{main.matrix}) whose characteristic polynomial is $p$.
\end{theorem}

The matrices of the form~(\ref{main.matrix}) are called the \emph{Schwarz matrices} after H.~Schwarz\footnote{Schwarz considered matrices $(n,n)$th nonzero entries rather than $(1,1)$th as we do. In this work, we follow H.Wall who considered matrices~(\ref{main.matrix}) earlier than Schwarz.} who developed a method of transformation a given nonderogatory matrix with the characteristic polynomial satisfying~(\ref{Hurw.det.ineq}) to the form~(\ref{main.matrix}) (see~\cite{Schwarz}).

Theorem~\ref{Theorem.Schwarz.matrix.Wall} provides a solution of somewhat direct and inverse problems for tridiagonal matrices of the form~(\ref{main.matrix}). These problems, however, are not spectral and concern properties of the characteristic polynomial of $J_n$. Nevertheless, their solution is important to solving spectral direct and especially inverse problems for the matrices of the form~(\ref{main.matrix}). Thus it makes sense to give a bit more detailed explanation of Theorem~\ref{Theorem.Schwarz.matrix.Wall}.

So given a polynomial $p$ defined by~(\ref{main.poly}) and satisfying the inequalities~(\ref{Hurw.det.ineq}), the matrix $J_n$ such that $p(z)=\det(zE_n-J_n)$ can be reconstruct, for instance, by formul\ae~(\ref{Formulae.1}). However, one can also run a Sturm algorithm as it was noted in~\cite{Schwarz}.

Indeed, let
$$%\begin{equation*}
\displaystyle f_0(z):=p(z)\qquad\mathrm{and}\qquad f_{1}(z):=\frac{q(z)}{b_0},
$$%\end{equation*}
where $q(z)$ is defined in~(\ref{poly.q}). The polynomials $f_0$ and $f_1$ are monic, and
$f_1$ is even or odd by construction. Now we construct a sequence of monic polynomials $f_k$,
$\deg{f_k}=n-k$, by the following process
$$%\begin{equation*}
\begin{array}{c}
b_1f_2(z):=f_0(z)-(z+b_0)f_1(z),\\
b_2f_3(z):=f_1(z)-zf_2(z),\\
\cdots\cdots\cdots\cdots\cdots\cdots\cdots\cdots\\
b_{n-2}f_{n-1}(z):=f_{n-3}(z)-zf_{n-2}(z),\\
b_{n-1}:=f_{n-2}(z)-zf_{n-1}(z).\\
\end{array}
$$%\end{equation*}
Thus, these equations give us all the entries $b_k$ of the matrix $J_n$ in~(\ref{main.matrix}). Moreover, the polynomials $f_k$, $k=1,\ldots,n$, are even or odd, and $f_k(z)$ is the characteristic polynomial of the principal submatrix of $J_n$ obtained by deleting first $k$ rows
and first $k$ columns.

We, finally, investigate the structure of the matrix~(\ref{main.matrix}) in detail. Let again
$p$ be its characteristic polynomial: $p(z)=\det(zE_n-J_n)$. Consider the following auxiliary matrix
$$%\begin{equation}\label{auxiliary.matrices}
A_n=
\left[\begin{array}{cccccc}
      0  & 1 &  0 &\dots&   0   & 0 \\
    -b_1 & 0 &1 &\dots&   0   & 0 \\
     0  &-b_2 & 0 &\dots&   0   & 0 \\
    \vdots&\vdots&\vdots&\ddots&\vdots&\vdots\\
     0  &  0  &  0  &\dots&0& 1\\
     0  &  0  &  0  &\dots&-b_{n-1}&0\\
\end{array}\right]
$$%\end{equation}
and its submatrix
$$%\begin{equation}\label{auxiliary.matrices}
A_{n-1}=
\left[\begin{array}{cccccc}
      0  & 1 &  0 &\dots&   0   & 0 \\
    -b_2 & 0 &1 &\dots&   0   & 0 \\
     0  &-b_3 & 0 &\dots&   0   & 0 \\
    \vdots&\vdots&\vdots&\ddots&\vdots&\vdots\\
     0  &  0  &  0  &\dots&0& 1\\
     0  &  0  &  0  &\dots&-b_{n-1}&0\\
\end{array}\right]
$$%\end{equation}
obtained from $A_n$ by deleting its first row and column. It is easy to see
that
$$
p(z)=\det(zE_n-J_n)=\det(zE_n-A_n)+b_0\det(zE_{n-1}-A_{n-1})
$$
It is also clear that if $p(z)=p_0(z^2)+zp_1(z^2)$, where $p_0(u)$
and $p_1(u)$ are the even and odd parts of $p$, respectively, then

\noindent for $n=2l$,
\begin{equation}\label{even.odd.parts.even.degree}
p_0(z^2)=\det(zE_n-A_n)\quad\textrm{and}\quad zp_1(z^2)=b_0\det(zE_{n-1}-A_{n-1})
\end{equation}

\noindent for $n=2l+1$,
$$
zp_1(z^2)=\det(zE_n-A_n)\quad\textrm{and}\quad p_0(z^2)=b_0\det(zE_{n-1}-A_{n-1})
$$

These formul\ae~imply the following simple fact.
\begin{prepos}\label{proposition.1}
Let the polynomial $p(z)=p_0(z^2)+zp_1(z^2)$ be the characteristic polynomial of the matrix $J_n$ given in~(\ref{main.matrix}), $p(z)=\det(zE_n-J_n)$. Then the~polynomial
$q(z)=(-1)^{\left[\frac{n+1}2\right]}[p_0(-z^2)+(-1)^nzp_1(-z^2)]$ is the characteristic polynomial of the matrix
\begin{equation}\label{auxiliary.matrix.in.propos}
\left[\begin{array}{cccccc}
     b_0 &  1  &  0  &\dots&   0   & 0 \\
     b_1 &  0  &  1  &\dots&   0   & 0 \\
      0  & b_2 &  0  &\dots&   0   & 0 \\
    \vdots&\vdots&\vdots&\ddots&\vdots&\vdots\\
      0  &  0  &  0  &\dots&   0   & 1\\
      0  &  0  &  0  &\dots&b_{n-1}& 0\\
\end{array}\right]
\end{equation}
\end{prepos}
\begin{proof}
We prove the proposition for $n=2l$. For $n=2l+1$, it can be established analogously.

So let $n=2l$. Then $\left[\frac{n+1}2\right]=l$, and the polynomial $q$ has the form
$$
q(z)=(-1)^lp_0(-z^2)+(-1)^lzp_1(-z^2)
$$
Using formul\ae~(\ref{even.odd.parts.even.degree}) one can obtain
$$
(-1)^lp_0(-z^2)=(-1)^l\det(-izE_n-A_{n})=\det(zE_n-iA_{n})=\det(zE_n-B_{n}),
$$
where
$$%\begin{equation}\label{auxiliary.matrices}
B_n=
\left[\begin{array}{cccccc}
      0  & 1 &  0 &\dots&   0   & 0 \\
    b_1 & 0 &1 &\dots&   0   & 0 \\
     0  &b_2 & 0 &\dots&   0   & 0 \\
    \vdots&\vdots&\vdots&\ddots&\vdots&\vdots\\
     0  &  0  &  0  &\dots&0& 1\\
     0  &  0  &  0  &\dots&b_{n-1}&0\\
\end{array}\right]\ .
$$%\end{equation}
Here we used the fact (see e.g.~\cite[Chapter~II]{KreinGantmaher}) that the characteristic polynomial
of any tridiagonal matrix does not depend on $(i+1,i)$th and $(i,i+1)$th entries separately but on their product, so the matrices $iA_n$ and $B_n$ have the same characteristic polynomial.

Analogously we have
$$
\begin{array}{c}
(-1)^lzp_1(-z^2)=i(-1)^lb_0\det(-izE_n-A_{n-1})=\\
-b_0\det(zE_n-iA_{n-1})=-b_0\det(zE_n-B_{n-1}),
\end{array}
$$
where
$$%\begin{equation}\label{auxiliary.matrices}
B_{n-1}=
\left[\begin{array}{cccccc}
      0  & 1 &  0 &\dots&   0   & 0 \\
    b_2 & 0 &1 &\dots&   0   & 0 \\
     0  &b_3 & 0 &\dots&   0   & 0 \\
    \vdots&\vdots&\vdots&\ddots&\vdots&\vdots\\
     0  &  0  &  0  &\dots&0& 1\\
     0  &  0  &  0  &\dots&b_{n-1}&0\\
\end{array}\right]\ .
$$%\end{equation}
Thus, we get
$$
q(z)=\det(zE_n-B_{n})-b_0\det(zE_n-B_{n-1}),
$$
so $q(z)$ is the characteristic polynomial of the matrix~(\ref{auxiliary.matrix.in.propos}).
\end{proof}

\setcounter{equation}{0}

%%%%%%%%%%%%%%%%%%%%%%%%%%%%%%%%%%%%%%%%%%%%%%%%%%%%%%%%%%%%%%%%%%%%%%%%%%%%%%%%%%%%%%%
\section{Some solved direct and inverse spectral problems for the Schwarz matrices}\label{section:solved}
%%%%%%%%%%%%%%%%%%%%%%%%%%%%%%%%%%%%%%%%%%%%%%%%%%%%%%%%%%%%%%%%%%%%%%%%%%%%%%%%%%%%%%%

In the previous section, we described properties of the characteristic polynomials of
the Schwarz matrices of the form~(\ref{main.matrix}) and recall methods of reconstruction
such matrices from their characteristic polynomials. However, we are interested in direct
and inverse \emph{spectral} problems of the Schwarz matrices.

It is natural to study a dependence of the spectrum of the matrix $J_n$ given in~(\ref{main.matrix}) in terms of \textit{signs} of the entries $b_k$ of this matrix.
Since we have the relations~(\ref{Formulae.1}) between the entries $b_k$ of the matrix
$J_n$ and the coefficients $a_j$ of its characteristic polynomials, it makes sense to use
the results of the theory of root location of polynomials which use signs of Hurwitz minors.

The most
known such result is the Hurwitz theorem stating that a real polynomial $p(z)$ given by~(\ref{main.poly}) has all its zeroes in the open left half-plane of the complex plane
if and only if its Hurwitz minors~(\ref{delta}) are \emph{positive}. We recall that
a real polynomial is called \emph{Hurwitz stable} if all its zeroes lie in the open left
half-plane.

Using the Hurwitz theorem H.~Wall established the following fact in~\cite[p.314]{Wall'45} (see also~\cite{Schwarz}).
\begin{theorem}[Wall]\label{Theorem.Routh.form}
The Schwarz matrix $J_n$ given in~(\ref{main.matrix}) has all its eigenvalues in the open left half-plane if all the entries $b_k$ are positive. Conversely, given a sequence of complex numbers
$\lambda_1,\ldots,\lambda_n$ with negative real parts, there exists a unique Schwarz matrix $J_n$ of the form~(\ref{main.matrix}) with $b_k>0$, $k=0,1,\ldots,n-1$, such that\footnote{$\sigma(J_n)$ denotes the spectrum of the matrix $J_n$.} $\sigma(J_n)=\{\lambda_1,\ldots,\lambda_n\}$.
\end{theorem}

Thus this theorem solves the direct and inverse spectral problems for stable Schwarz matrices, that is, the Schwarz matrices with positive $b_k$, that are sometimes called \textit{Routh canonical forms} (see e.g.~\cite{Puri_Weygandt,Datta2}). Note that Wall's work~\cite{Wall'45} where he considered Hurwitz stable polynomials as characteristic polynomials of matrices~(\ref{main.matrix}) with positive $b_k$ appeared earlier than the paper~\cite{Schwarz} by H.~Schwarz.

%\vspace{1mm}

Next result regarding eigenvalue location of the matrix~(\ref{main.matrix}) is based on the
so-called Routh-Hurwitz theorem established by Gantmacher in~\cite[Theorem~4, p.~230]{Gantmakher}.
\begin{theorem}[Routh-Hurwitz]\label{Theorem.Routh-Hurwitz}
Let the polynomial $p$ be defined in~(\ref{main.poly}) and satisfy~(\ref{Hurw.det.ineq}).
The number $m$ of roots of $p$ which lie in the open right half-plane is given by the formula
$$%\begin{equation*}
\displaystyle m=v\left(1,\Delta_1(p),\frac{\Delta_2(p)}{\Delta_1(p)},\frac{\Delta_3(p)}{\Delta_2(p)},\ldots,\frac{\Delta_{n}(p)}{\Delta_{n-1}(p)}\right)
$$%\end{equation*}
or equivalently
$$%\begin{equation*}
\displaystyle m=v(1,\Delta_1(p),\Delta_3(p),\ldots)+v(1,\Delta_2(p),\Delta_4(p),\ldots),
$$%\end{equation*}
where $\Delta_j(p)$ are the Hurwitz determinants of $p$, and $v(c_0,c_1,\ldots,c_l)$ denotes the number of sign changes in the sequence $[c_0,c_1,\ldots,c_l]$.
\end{theorem}

\begin{remark}
Note that in Theorem~\ref{Theorem.Routh-Hurwitz}, all Hurwitz determinants of the polynomial $p$
are nonzero by assumption, so we use the standard calculation of the sign changes in the sequences of the Hurwitz determinants. However, this theorem is also true in the case when some of Hurwitz determinants of $p$
equal zero~\cite[\S8, p.~235]{Gantmakher} (see also~\cite{Holtz_Tyaglov} or
comments to Theorem~\ref{Theorem.general.Hurwitz.criterion} on the page~\pageref{Frobenius}).
\end{remark}

Using Theorem~\ref{Theorem.Routh-Hurwitz} and formul\ae~(\ref{Formulae.1}) one can easily obtain
the following result due to Schwarz~\cite[Satz~5]{Schwarz} which also follows from Theorem D and
formul\ae~$(2.1)$ of Wall's work~\cite{Wall'45}.
\begin{theorem}\label{Theorem.Schwarz.Routh-Hurwitz}
Given a real matrix $J_n$ as in~(\ref{main.matrix}), the number of negative terms
in the sequence
\begin{equation}\label{entries.sequence}
b_0,\  b_0b_1,\  b_0b_1b_2,\ldots, \ b_0b_1\cdots b_{n-1}
\end{equation}
is equal to the number of eigenvalues of $J_n$ in the open right half-plane of the complex plane.
\end{theorem}

This theorem uses sign patterns of the entries of the matrix $J_n$ to localize distributions of
its eigenvalues. So this result can be identified as the solution of the direct spectral problem
of the matrix $J_n$. The inverse spectral problem is somewhat trivial in light of Theorems~\ref{Theorem.Schwarz.matrix.Wall} and~\ref{Theorem.Routh-Hurwitz} and formul\ae~(\ref{Formulae.1}).
\begin{theorem}\label{Theorem.Schwarz.Routh-Hurwitz.inverse}
Let $\lambda_1,\ldots,\lambda_n$ be a sequence of complex numbers with $m$ numbers in the open right half-plane and $n-m$ numbers in the open left half-plane such that the polynomial $p(z)=\prod\limits_{i=1}^n(z-\lambda_i)$ satisfies the inequalities~(\ref{Hurw.det.ineq}).
There exists a unique matrix $J_n$ of the form~(\ref{main.matrix}) such that the number of negative terms in the sequence~(\ref{entries.sequence}) constructed with the entries of $J_n$ equals $m$
and $\sigma(J_n)=\{\lambda_1,\ldots,\lambda_n\}$.
\end{theorem}

Note that in Theorem~\ref{Theorem.Routh.form} we did not need to suppose the polynomial $p(z)=\prod\limits_{i=1}^n(z-\lambda_i)$ to satisfy~(\ref{Hurw.det.ineq}), because all
Hurwitz stable polynomials automatically satisfy them by the Hurwitz theorem we mentioned
above (see also Theorem~\ref{Theorem.Routh-Hurwitz} and remark after it).

\vspace{1mm}

Theorem~\ref{Theorem.Routh.form} deals with the Schwarz matrices with positive $b_k$, so it is natural to study the Schwarz matrices~(\ref{main.matrix}) with all negative $b_k$. This problem
was solved by O.~Holtz in~\cite[Corollary~2]{H4}, where she proved the following\footnote{It is worth to note that there is a mistake in the proof of the main theorem, Theorem~1, in~\cite{H4}. However, this mistake can be easily corrected, while the statement of the theorem is valid.}.
\begin{theorem}\label{Theorem.Holtz}
Let the matrix $J_n$ be defined in~(\ref{main.matrix}) with $b_k<0$, $k=0,\ldots,n-1$. Then its
eigenvalues~$\lambda_i$ are simple real and satisfy the inequalities
\begin{equation}\label{Theorem.Holtz.ineq}
\lambda_1>-\lambda_2>\lambda_3>\cdots>(-1)^{n-1}\lambda_n>0.
\end{equation}

Conversely, for any sequence of real numbers $\lambda_1,\ldots,\lambda_n$ distributed as in~(\ref{Theorem.Holtz.ineq}), there exists a unique matrix $J_n$ of the form~(\ref{main.matrix}) with $b_k<0$, $k=0,\ldots,n-1$, such that $\sigma(J_n)=\{\lambda_1,\ldots,\lambda_n\}$.
\end{theorem}

This theorem was proved  in~\cite{Bebiano'11} by a technique different from one used in~\cite{H4}. However, it can be proved easily using properties of generalized Hurwitz
polynomials~\cite{Tyaglov.gen.Hurw} (see Remark~\ref{remark.proof.Holtz}). We just note that in Theorem~\ref{Theorem.Holtz}, there is no requirement for the polynomial $p(z)=\prod\limits_{i=1}^n(z-\lambda_i)$ to satisfy the inequalities~(\ref{Hurw.det.ineq}). As we will show, the polynomials with the distribution of zeroes as in~(\ref{Theorem.Holtz.ineq}) automatically satisfy~(\ref{Hurw.det.ineq}), since they are dual (in some sense\footnote{See Theorem~\ref{Theorem.connection.Hurwitz.self-interlacing}.}) to Hurwitz stable polynomials.

Finally, we should mention that in~\cite{Bebiano'11}, there was an attempt to solve direct and inverse problems for the matrix~(\ref{main.matrix}) with $b_0<0$, $b_1>0$ and $b_k<0$ for $k=2,\ldots,n-1$. However, their result is incorrect. We include the correct version of their Theorem~9 as an example of application of our results (see Theorem~\ref{Theorem.Bebiano.matrix}).

\setcounter{equation}{0}

%%%%%%%%%%%%%%%%%%%%%%%%%%%%%%%%%%%%%%%%%%%%%%%%%%%%%%%%%%%%%%%%%%%%%%%%%%%%%%%%%%%%%%%%%%%%%%%
\section{Generalized Hurwitz polynomials, basic properties}\label{section:generalized.Hurwitz}
%%%%%%%%%%%%%%%%%%%%%%%%%%%%%%%%%%%%%%%%%%%%%%%%%%%%%%%%%%%%%%%%%%%%%%%%%%%%%%%%%%%%%%%%%%%%%%%%%

In this section, we define (almost) generalized Hurwitz polynomials~\cite{Tyaglov.gen.Hurw} and review their basic property, which will be helpful to study spectral problems of the Schwarz matrices.

\begin{definition}\label{def.general.Hurw.poly}
A real polynomial $p$ is called \textit{generalized Hurwitz
polynomial} of type~I of order $\varkappa$, where\footnote{Here $[\alpha]$ denotes the maximal integer not exceeding $\alpha$.} $1\leqslant
\varkappa\leqslant\left[\frac{n+1}2\right]$, if it has exactly $\varkappa$ zeroes
in the \emph{closed} right half-plane, all of which are nonnegative and
simple:
$$%\begin{equation*}%\label{Hurwitz.poly.def}
0\leqslant\mu_1<\mu_2<\cdots<\mu_{\varkappa},
$$%\end{equation*}
such that $p(-\mu_i)\neq0$, $i=2,\ldots,\varkappa$, $p(-\mu_1)\neq0$ if $\mu_1>0$, and $p$ has an odd
number of zeroes, counting multiplicities, on each interval
$(-\mu_{\varkappa},-\mu_{\varkappa-1}),\ldots,(-\mu_3,-\mu_2)$, $(-\mu_2,-\mu_1)$.
Moreover, the number of zeroes of $p$ on the interval~$(-\mu_1,0)$
(if any) is even, counting multiplicities. The other \emph{real} zeroes
lie on the interval $(-\infty,-\mu_{\varkappa})$: an odd number of zeroes,
counting multiplicities, when $n=2l$, and an even number of
zeroes, counting multiplicities, when $n=2l+1$. All nonreal zeroes
of $p$ (if any) are located in the open left half-plane of the
complex plane.
\end{definition}

Thus, the order $\varkappa$ of a generalized Hurwitz polynomial of type~I indicates the number of its zeroes
in the closed right half-plane. Moreover, the zeroes of a generalized Hurwitz polynomial
in the closed right half-plane structure the distribution of its negative zeroes, so not every real polynomial
with only real simple zeroes in the closed right half-plane is generalized Hurwitz. The generalized
Hurwitz polynomials of type~I of order $0$ are obviously Hurwitz stable polynomials, since they have no zeroes in the open right half-plane.

The generalized Hurwitz polynomials of type~II is a generalization of real polynomials with zeroes in the
open right half plane.
\begin{definition}\label{def.general.Hurw.poly.type.II}
A real polynomial $p(z)$ is generalized Hurwitz of type~II if the polynomial
$p(-z)$ is generalized Hurwitz of type~I.
\end{definition}

It is clear that all results obtained for the generalized Hurwitz
polynomials of type~I can be easily reformulated for  the
generalized Hurwitz polynomials of type~II. Thus we formulate all results
in this section only for generalized Hurwitz polynomials of type~I.

The main fact about generalized Hurwitz polynomials we use in this paper is the following theorem
established in~\cite{Tyaglov.gen.Hurw}.
\begin{theorem}[Generalized Hurwitz theorem]\label{Theorem.general.Hurwitz.criterion}
The polynomial $p$ given in~(\ref{main.poly}) is
generalized Hurwitz if and~only~if
\begin{equation}\label{general.Hurwitz.main.det.noneq}
\Delta_{n-1}(p)>0,\ \Delta_{n-3}(p)>0,\ \Delta_{n-5}(p)>0,\ldots
\end{equation}
The order $\varkappa$ of the polynomial $p$ equals
\begin{equation}\label{general.Hurwitz.poly.10}
\varkappa={\rm V^{F}}(\Delta_{n}(p),\Delta_{n-2}(p),\ldots,1)\qquad\mathrm{if}\quad
p(0)\neq0,
\end{equation}
or
\begin{equation}\label{general.Hurwitz.poly.11}
\varkappa={\rm V^{F}}(\Delta_{n-2}(p),\Delta_{n-4}(p),\ldots,1)+1\qquad\mathrm{if}\quad
p(0)=0,
\end{equation}
where $V^{F}(c_1,\ldots,c_n)$ denotes the number of sign changes in the sequence $\{c_1,\ldots,c_n\}$ calculated
in accordance with the Frobenius rule of signs.
\end{theorem}

Recall that the Frobenius rule of signs is the following.

\noindent\textbf{Frobenius rule of signs}~\cite{Frobenius}\label{Frobenius} (see also~\cite[Ch.~X, \S10]{Gantmakher.1} and \cite[Ch.~2]{Holtz_Tyaglov}). \emph{Given a sequence of real numbers $\{c_1,\ldots,c_n\}$, where $c_1c_n\neq0$,
if, for some $i$ and $j$ $(0\leqslant i\leqslant j)$,}
$$%$$%\begin{equation*}
c_i\neq0,\quad c_{i+1}=c_{i+2}=\cdots=c_{i+j}=0,\quad c_{i+j+1}\neq0
$$%$$%\end{equation*}
\emph{then the number $V^{F}(c_1,\ldots,c_n)$ of Frobenius sign changes must be calculated by assigning signs as follows:}
$$%$$%\begin{equation*}
\sgn c_{i+\nu}=(-1)^{\frac{\nu(\nu-1)}2}\sgn c_i, \quad \nu=1,2,\ldots,j.
$$%$$%\end{equation*}

The Frobenius rule of signs was introduced by Frobenius~\cite{Frobenius} for calculating the number of sign changes in a sequence of Hankel minors. For details, see~\cite{Holtz_Tyaglov}.

Since we consider only polynomials with nonzero Hurwitz determinants in this work, in the rest of the paper the number of Frobenius sign changes $V^{F}$ will be changed by the standard number of sign changes~$v$, and the formula~(\ref{general.Hurwitz.poly.11}) will not be used, since
$\Delta_n(p)=0$ if $p(0)=0$.

By~(\ref{general.Hurwitz.poly.10})--(\ref{general.Hurwitz.poly.11}), $\varkappa=0$ if and only if $\Delta_{n-2k}(p)>0$, $k=0,1,\ldots,\left[\frac {n-1}2\right]$. As we mentioned above, the generalized Hurwitz polynomials with $\varkappa=0$ are Hurwitz stable polynomials. Thus, Theorem~\ref{Theorem.general.Hurwitz.criterion} implies that a real polynomial $p$ of degree $n$ is Hurwitz stable if and only if $\Delta_j(p)>0$, $j=1,\ldots,n$. This is exactly the Hurwitz stability criterion.

On the other side, the formul\ae~(\ref{general.Hurwitz.poly.10})--(\ref{general.Hurwitz.poly.11}) imply
that $\varkappa=\left[\frac{n+1}2\right]$ with $p(0)\neq0$ if and only if
\begin{equation}\label{self-interlacing.inequalities}
(-1)^{d}\Delta_{n}(p)>0,(-1)^{d}\Delta_{n-2}(p)>0,\ldots,\quad\mathrm{where}\quad d=\left[\frac {n+1}2\right].
\end{equation}
In this case, the generalized Hurwitz polynomial $p$ of type~I has neither nonreal nor multiple zeroes, so its zeroes are real and simple. Moreover,
they are distributed as follows:
\begin{equation}\label{self-interlacing.zero.distribution.I.type}
0<\lambda_1<-\lambda_2<\lambda_3<\ldots<(-1)^{n-1}\lambda_n.
\end{equation}
\begin{definition}\label{def.self-interlacing.I}
A real polynomial whose zeroes are distributed as in~(\ref{self-interlacing.zero.distribution.I.type}) is called \textit{self-interlacing} of type~I.
\end{definition}

Analogously to the general case, we introduce the self-interlacing polynomials of type~II.
\begin{definition}\label{def.self-interlacing.II}
 A polynomial $p(z)$ is called self-interlacing of type~II if $p(-z)$ is self-interlacing of type~I, or equivalently if its zeroes are distributed as follows:
\begin{equation}\label{self-interlacing.zero.distribution.II.type}
0<-\lambda_1<\lambda_2<-\lambda_3<\ldots<(-1)^{n}\lambda_n.
\end{equation}
\end{definition}

From Definitions~\ref{def.self-interlacing.I}--\ref{def.self-interlacing.II} it is easy to see that a real polynomial $p(z)$ is self-interlacing (of type~I or~II) if and only if it has real and simple zeroes which interlace the zeroes of the polynomial $p(-z)$.

If we return now to Theorem~\ref{Theorem.Holtz}, we will see from~(\ref{Theorem.Holtz.ineq}) that the characteristic polynomials of the matrices~(\ref{main.matrix}) with all $b_k<0$ are self-interlacing polynomials: of type~I for odd $n$, and of type~II for even~$n$.

In~\cite{Tyaglov.gen.Hurw}, there was also established the following important fact about the
relation (in fact, duality) between Hurwitz stable and self-interlacing polynomials. We will use
this fact later to reveal a relation between Theorems~\ref{Theorem.Routh.form} and~\ref{Theorem.Holtz} (see Remark~\ref{remark.proof.Holtz}).
\begin{theorem}\label{Theorem.connection.Hurwitz.self-interlacing}
A polynomial $p(z)=p_0(z^2)+zp_1(z^2)$ is self-interlacing of type~I
if and only if the polynomial $q(z)=p_0(-z^2)-zp_1(-z^2)$ is
Hurwitz stable, where $p_0(u)$ and $p_1(u)$ are the even and odd parts of $p$, respectively (see~(\ref{even.part})--(\ref{odd.part})).
\end{theorem}

\noindent Indeed, there can be established a more general fact.
\begin{theorem}[\cite{Tyaglov.gen.Hurw}]\label{Theorem.connection.Hurwitz.gen.Hurw}
A polynomial $p(z)=p_0(z^2)+zp_1(z^2)$, $p(0)\neq0$, is generalized Hurwitz of order $\varkappa$ of type~I (type~II) if and only if the polynomial $q(z)=p_0(-z^2)-zp_1(-z^2)$ is generalized Hurwitz of order $\left[\frac{n+1}2\right]-\varkappa$ of type~I (type~II), where $n=\deg p$.
\end{theorem}

\begin{remark}\label{remark.type.II.0}
\emph{From Definition~\ref{def.general.Hurw.poly.type.II} and Theorem~\ref{Theorem.connection.Hurwitz.gen.Hurw} it is clear that if $p(z)=p_0(z^2)+zp_1(z^2)$ is generalized Hurwitz of type~I, then $p(z)=p_0(-z^2)+zp_1(-z^2)$ is generalized Hurwitz of type~II.}
\end{remark}
\begin{remark}\label{remark.proof.Holtz}
\emph{Theorem~\ref{Theorem.Routh.form}, Proposition~\ref{proposition.1}, Theorem~\ref{Theorem.connection.Hurwitz.self-interlacing} and Remark~\ref{remark.type.II.0} imply Theorem~\ref{Theorem.Holtz}. Conversely, Wall's Theorem~\ref{Theorem.Routh.form} can be obtained
from Theorem~\ref{Theorem.Holtz}, Theorem~\ref{Theorem.connection.Hurwitz.self-interlacing}, and Proposition~\ref{proposition.1} taking
into account Remark~\ref{remark.type.II.0}.}
\end{remark}

Finally, let us introduce the so-called almost generalized Hurwitz polynomials.
\begin{definition}\label{def.almost.general.Hurw.poly}
A real polynomial $p(z)$ called almost generalized Hurwitz of order $\varkappa$ of type~I (type~II) if the polynomial $zp(z)$ is generalized Hurwitz of order $\varkappa+1$ of type~I (resp. type~II).
\end{definition}
\begin{remark}\label{remark.almost.gen}
\emph{Note that any almost generalized Hurwitz polynomial of order~$0$ of type~I is a Hurwitz stable polynomial, while any almost generalized Hurwitz polynomial of type~I of degree $2l$ and of order $l$ is a self-interlacing polynomial of type~II. Also any almost generalized Hurwitz polynomial of type~II of degree $2l+1$ and of order $l$ is a self-interlacing polynomial of type~I.}
\end{remark}

For almost generalized Hurwitz polynomials we have the following basic theorem analogous to Theorem~\ref{Theorem.general.Hurwitz.criterion} (see~\cite{Tyaglov.gen.Hurw}).

\begin{theorem}\label{Theorem.almost.general.Hurwitz.criterion}
The polynomial $p$ given in~(\ref{main.poly}) is
generalized Hurwitz if and only if
\begin{equation}\label{almost.general.Hurwitz.main.det.noneq}
\Delta_{n}(p)>0,\ \Delta_{n-2}(p)>0,\ \Delta_{n-4}(p)>0,\ldots
\end{equation}
The order $\varkappa$ of the polynomial $p$ equals
\begin{equation}\label{almost.general.Hurwitz.poly.10}
\varkappa={\rm V^{F}}(\Delta_{n-1}(p),\Delta_{n-3}(p),\ldots,1).
\end{equation}
where $V^{F}(c_1,\ldots,c_n)$ denotes the number of sign changes in the sequence $\{c_1,\ldots,c_n\}$ calculated in accordance with the Frobenius rule of signs.
\end{theorem}

Note that almost generalized Hurwitz polynomials do not vanish at zero, so they order
equal the number of positive simple zeroes. One can easily describe the distribution of zeroes
of almost generalized Hurwitz polynomials from Definitions~\ref{def.general.Hurw.poly} and~\ref{def.almost.general.Hurw.poly}. Moreover, if a real polynomial is generalized Hurwitz
and almost generalized Hurwitz simultaneously, then it is Hurwitz stable.

\setcounter{equation}{0}
%%%%%%%%%%%%%%%%%%%%%%%%%%%%%%%%%%%%%%%%%%%%%%%%%%%%%%%%%%%%%%%%%%%%%%%%%%%%%%%%%%%%%%%%%%%%%%%%%%%%%
\section{Direct and inverse spectral problems for Schwarz matrices}\label{section:spectral.problems}
%%%%%%%%%%%%%%%%%%%%%%%%%%%%%%%%%%%%%%%%%%%%%%%%%%%%%%%%%%%%%%%%%%%%%%%%%%%%%%%%%%%%%%%%%%%%%%%%%%%%%%

Let us again consider the Schwarz matrix
\begin{equation}\label{Schwarz.matrix}
J_n=
\left[\begin{array}{cccccc}
    -b_0 & 1 &  0 &\dots&   0   & 0 \\
    -b_1 & 0 &1 &\dots&   0   & 0 \\
     0  &-b_2 & 0 &\dots&   0   & 0 \\
    \vdots&\vdots&\vdots&\ddots&\vdots&\vdots\\
     0  &  0  &  0  &\dots&0& 1\\
     0  &  0  &  0  &\dots&-b_{n-1}&0\\
\end{array}\right]
\end{equation}
with all $b_k$ nonzero, and denote by $p(z)$ its characteristic polynomial, that is, $p(z)=\det(zE_n-J_n)$. From formul\ae~(\ref{Formulae.1})
it is easy to get the following
\begin{equation}\label{Formulae.odd}
b_0=\Delta_1(p),\quad b_{2j-1}b_{2j}=\frac{\Delta_{2j-3}(p)\Delta_{2j+1}(p)}{\Delta_{2j-1}^2(p)},\qquad j=1,\ldots,\left[\frac{n-1}2\right],
\end{equation}
and
\begin{equation}\label{Formulae.even}
b_{2j}b_{2j+1}=\frac{\Delta_{2j-2}(p)\Delta_{2j+2}(p)}{\Delta_{2j}^2(p)},\qquad j=0,1,\ldots,\left[\frac{n-2}2\right],
\end{equation}
where $\Delta_{-2}(p)\equiv1$, and $[\alpha]$ denotes the maximal integer not exceeding $\alpha$.

From the formul\ae~(\ref{Formulae.odd})--(\ref{Formulae.even}) and from Theorems~\ref{Theorem.Schwarz.matrix.Wall} and~\ref{Theorem.general.Hurwitz.criterion} it is easy to obtain
the following fact.
\begin{theorem}\label{Theorem.gen.Hurw.Schwarz.even}
Let the matrix $J_n$ be given in~(\ref{Schwarz.matrix}), and $n=2l$. The characteristic polynomial
$p$ of the matrix $J_n$ is generalized Hurwitz of type~I if and only~if
\begin{equation}\label{Gen.Hurw.even.b.ineq.1}
b_0>0,\,b_1b_2>0,\,b_3b_4>0,\ldots,\,b_{n-3}b_{n-2}>0.
\end{equation}
The order $\varkappa$ of the polynomial $p$ is equal to the number of negative terms in the sequence
\begin{equation}\label{Gen.Hurw.even.b.ineq.2}
b_0b_1,\  b_0b_1b_2b_3,\  b_0b_1b_2b_3b_4b_5,\ldots, \ b_0b_1\cdots b_{n-1}.
%b_0b_1,\  \prod_{i=0}^3b_i,\  \prod_{i=0}^5b_i,\ldots, \ \prod_{i=0}^{n-1}b_i.
\end{equation}

Conversely, let $\lambda_1,\ldots,\lambda_n$ be a sequence of complex numbers such that the polynomial $p(z)=\prod_{k=1}^n(z-\lambda_k)$ is generalized Hurwitz of type~I of order $\varkappa$ and satisfies the inequalities~(\ref{Hurw.det.ineq}). Then there exists a unique Schwarz matrix $J_n$ of the form~(\ref{Schwarz.matrix}) with entries $b_k$ satisfying~(\ref{Gen.Hurw.even.b.ineq.1})
such that the number of negative terms in the sequence~(\ref{Gen.Hurw.even.b.ineq.2}) is equal to~$\varkappa$, and $\sigma(J_n)=\{\lambda_1,\ldots,\lambda_n\}$.
\end{theorem}
\begin{proof}
Let $p$ be the characteristic polynomial of the matrix $J_n$. It satisfies~(\ref{Hurw.det.ineq})
by assumption. According to Theorem~\ref{Theorem.general.Hurwitz.criterion}, it is generalized
Hurwitz of type~I if and only if $\Delta_{2i-1}(p)>0$ for $i=1,\ldots,l$.
By~(\ref{Formulae.odd}), these inequalities are equivalent to~(\ref{Gen.Hurw.even.b.ineq.1}).
Furthermore, from~(\ref{Formulae.even}) we have
\begin{equation}\label{Theorem.gen.Hurw.Schwarz.even.proof.1}
\prod\limits_{k=0}^{2i-1}b_k=\frac{\Delta_{2i}(p)}{\Delta_{2i-2}(p)},\qquad i=1,\ldots,l.
\end{equation}
By Theorem~\ref{Theorem.general.Hurwitz.criterion}, the order of the generalized Hurwitz polynomial $p$ is equal to the number of sign changes in the sequence $\Delta_{2}(p)$, $\Delta_{4}(p)$, \ldots, $\Delta_{2l}(p)$. But from~(\ref{Theorem.gen.Hurw.Schwarz.even.proof.1}) we obtain that each sign change in this sequence corresponds to a negative number in the sequence~(\ref{Gen.Hurw.even.b.ineq.2}).

Conversely, if the complex numbers $\lambda_1,\ldots,\lambda_n$ are such that the polynomial $p(z)=\prod_{k=1}^n(z-\lambda_k)$ is generalized Hurwitz of type~I of order $\varkappa$ satisfying the inequalities~(\ref{Hurw.det.ineq}), then by Theorems~\ref{Theorem.Schwarz.matrix.Wall} and~\ref{Theorem.general.Hurwitz.criterion} and by formul\ae~(\ref{Formulae.odd})--(\ref{Formulae.even}), there exists a unique matrix $J_n$ of the form~(\ref{main.matrix}) satisfying the inequalities~(\ref{Gen.Hurw.even.b.ineq.1}) and with~$\varkappa$ negative numbers in the sequence~(\ref{Gen.Hurw.even.b.ineq.2}) such that its
characteristic polynomial~is~$p$.
\end{proof}

Analogously, using formul\ae~(\ref{Formulae.odd})--(\ref{Formulae.even}) and Theorems~\ref{Theorem.Schwarz.matrix.Wall},~\ref{Theorem.general.Hurwitz.criterion},
and~\ref{Theorem.almost.general.Hurwitz.criterion} one can easily establish the following theorems.
\begin{theorem}\label{Theorem.gen.Hurw.Schwarz.odd}
Let the matrix $J_n$ be given in~(\ref{Schwarz.matrix}), and $n=2l+1$. The characteristic
polynomial $p$ of the matrix $J_n$ is generalized Hurwitz of type~I~if~and~only~if
\begin{equation}\label{Gen.Hurw.odd.b.ineq.1}
b_0b_1>0,\,b_2b_3>0,\,b_4b_5>0,\ldots,\,b_{n-3}b_{n-2}>0.
\end{equation}
The order $\varkappa$ of the polynomial $p$ is equal to the number of negative terms in the sequence
\begin{equation}\label{Gen.Hurw.odd.b.ineq.2}
b_0,\  b_0b_1b_2,\  b_0b_1b_2b_3b_4,\ldots, \ b_0b_1\cdots b_{n-1}.
%b_0b_1,\  \prod_{i=0}^3b_i,\  \prod_{i=0}^5b_i,\ldots, \ \prod_{i=0}^{n-1}b_i.
\end{equation}

Conversely, let $\lambda_1,\ldots,\lambda_n$ be a sequence of complex numbers such that the polynomial $p(z)=\prod_{k=1}^n(z-\lambda_k)$ is generalized Hurwitz of type~I of order $\varkappa$ and satisfies the inequalities~(\ref{Hurw.det.ineq}). Then there exists a unique Schwarz matrix $J_n$ of the form~(\ref{Schwarz.matrix}) with entries $b_k$ satisfying~(\ref{Gen.Hurw.odd.b.ineq.1})
such that the number of negative terms in the sequence~(\ref{Gen.Hurw.odd.b.ineq.2}) is equal to~$\varkappa$, and $\sigma(J_n)=\{\lambda_1,\ldots,\lambda_n\}$.
\end{theorem}
\begin{theorem}\label{Theorem.almost.gen.Hurw.Schwarz.even}
Let the matrix $J_n$ be given in~(\ref{Schwarz.matrix}), and $n=2l$. The characteristic
polynomial $p$ of the matrix $J_n$ is almost generalized Hurwitz of type~I if and only~if
\begin{equation}\label{Al.Gen.Hurw.even.b.ineq.1}
b_0b_1>0,\,b_2b_3>0,\,b_4b_5>0,\ldots,\,b_{n-2}b_{n-1}>0.
\end{equation}
The order $\varkappa$ of the polynomial $p$ is equal to the number of negative terms in the sequence
\begin{equation}\label{Al.Gen.Hurw.even.b.ineq.2}
b_0,\  b_0b_1b_2,\  b_0b_1b_2b_3b_4,\ldots, \ b_0b_1\cdots b_{n-2}.
%b_0b_1,\  \prod_{i=0}^3b_i,\  \prod_{i=0}^5b_i,\ldots, \ \prod_{i=0}^{n-1}b_i.
\end{equation}

Conversely, let $\lambda_1,\ldots,\lambda_n$ be a sequence of complex numbers such that the polynomial $p(z)=\prod_{k=1}^n(z-\lambda_k)$ is almost generalized Hurwitz of type~I of order $\varkappa$ and satisfies the inequalities~(\ref{Hurw.det.ineq}). Then there exists a unique Schwarz matrix $J_n$ of the form~(\ref{Schwarz.matrix}) with entries $b_k$ satisfying~(\ref{Al.Gen.Hurw.even.b.ineq.1})
such that the number of negative terms in the sequence~(\ref{Al.Gen.Hurw.even.b.ineq.2}) is equal to~$\varkappa$, and $\sigma(J_n)=\{\lambda_1,\ldots,\lambda_n\}$.
\end{theorem}
\begin{theorem}\label{Theorem.almost.gen.Hurw.Schwarz.odd}
Let the matrix $J_n$ be given in~(\ref{Schwarz.matrix}), and $n=2l+1$. The characteristic
polynomial$p$ of the matrix $J_n$ is almost generalized Hurwitz of type~I~if~and~only~if
\begin{equation}\label{Al.Gen.Hurw.odd.b.ineq.1}
b_0>0,\,b_1b_2>0,\,b_3b_4>0,\ldots,\,b_{n-2}b_{n-1}>0.
\end{equation}
The order $\varkappa$ of the polynomial $p$ is equal to the number of negative terms in the sequence
\begin{equation}\label{Al.Gen.Hurw.odd.b.ineq.2}
b_0b_1,\  b_0b_1b_2b_3,\  b_0b_1b_2b_3b_4b_5,\ldots, \ b_0b_1\cdots b_{n-2}.
%b_0b_1,\  \prod_{i=0}^3b_i,\  \prod_{i=0}^5b_i,\ldots, \ \prod_{i=0}^{n-1}b_i.
\end{equation}

Conversely, let $\lambda_1,\ldots,\lambda_n$ be a sequence of complex numbers such that the polynomial $p(z)=\prod_{k=1}^n(z-\lambda_k)$ is almost generalized Hurwitz of type~I of order $\varkappa$ and satisfies the inequalities~(\ref{Hurw.det.ineq}). Then there exists a unique Schwarz matrix $J_n$ of the form~(\ref{Schwarz.matrix}) with entries $b_k$ satisfying~(\ref{Al.Gen.Hurw.odd.b.ineq.1})
such that the number of negative terms in the sequence~(\ref{Al.Gen.Hurw.odd.b.ineq.2}) is equal to~$\varkappa$, and $\sigma(J_n)=\{\lambda_1,\ldots,\lambda_n\}$.
\end{theorem}

\begin{remark}\label{remark.type.II}
\emph{It is also easy to prove and formulate an analogous theorems for (almost) generalized Hurwitz polynomials of type~II. But it is not necessary, since if the characteristic polynomial of a matrix $J_n$ is (almost) generalized Hurwitz of type~II, then the characteristic polynomial of a matrix $-J_n$ is (almost) generalized Hurwitz of type~I. But changing the sign of the matrix will change,
in fact, just the sign of the entry $b_0$, since the characteristic polynomial of tridiagonal matrices depends on the products of the $(i,i+1)$th and $(i+1,i)$th entries rather than on these entries separately. So if we change their
signs simultaneously, this does not change the characteristic polynomial~\cite[Chapter~II]{KreinGantmaher}. Thus, if we have a matrix $J_n$ of the form~(\ref{Schwarz.matrix}) such that $b_0<0$ \textit{and} $b_0b_1<0$, we should consider the matrix $-J_n$ and apply one of Theorems~\ref{Theorem.gen.Hurw.Schwarz.even}--\ref{Theorem.almost.gen.Hurw.Schwarz.odd} (if any).}
\end{remark}

\setcounter{equation}{0}

%%%%%%%%%%%%%%%%%%%%%%%%%%%%%%%%%%%%%%%%%%%
\section{Examples}\label{section:examples}
%%%%%%%%%%%%%%%%%%%%%%%%%%%%%%%%%%%%%%%%%%%

In this section, we show how the results of the previous section can be used for certain sign patterns of the Schwarz matrix~(\ref{Schwarz.matrix}).

Consider the following matrix
\begin{equation}\label{Schwarz.matrix.2}
S_n=
\left[\begin{array}{ccccccccccc}
    -a   & 1   & 0 &\dots&  0  &   0   & 0 & 0 &\dots&   0   & 0\\
    -c_1 & 0   & 1 &\dots&  0  &   0   & 0 & 0 &\dots&   0   & 0\\
     0   &-c_2 & 0 &\dots&  0  &   0   & 0 & 0 &\dots&   0   & 0\\
    \vdots&\vdots&\vdots&\ddots&\vdots&\vdots&\vdots&\vdots&\ddots&\vdots&\vdots\\
     0   &  0  & 0 &\dots&-c_k &   0   & 1 & 0 &\dots&   0   & 0\\
     0   &  0  & 0 &\dots&  0  &c_{k-1}& 0 & 1 &\dots&   0   & 0\\
    \vdots&\vdots&\vdots&\ddots&\vdots&\vdots&\vdots&\vdots&\ddots&\vdots&\vdots\\
     0   &  0  & 0 &\dots&  0  &   0   & 0 & 0 &\dots&   0   & 1\\
     0   &  0  & 0 &\dots&  0  &   0   & 0 & 0 &\dots&c_{n-1}& 0\\
\end{array}\right],
\end{equation}
where $a\in\mathbb{R}\backslash\{0\}$, and $c_k>0$ for $k=1,\ldots,n-1$.

\begin{theorem}\label{Theorem.my.old}
Let $p(z)$ be the characteristic polynomial of the matrix $S_n$:
$$
p(z)=\det(zE_n-S_n).
$$
\begin{itemize}
\item If $n=2l+1$ and $k=2m+1$ or $n=2l$ and $k=2m$, then $p(z)$ is generalized Hurwitz of order $\varkappa=l-m$ of type~I (type~II) provided $a>0$ (resp. $a<0$).
\item If $n=2l+1$ and $k=2m$ or $n=2l$ and $k=2m-1$, then $p(z)$ is almost generalized Hurwitz of order $\varkappa=l-m$ of type~I (type~II) provided $a>0$ (resp. $a<0$).
\end{itemize}
\end{theorem}
\begin{proof}
Without loss of generality suppose that $a>0$ (see Remark~\ref{remark.type.II}).
From the conditions of the theorem and from the formul\ae~(\ref{Formulae.1}), we obtain that the
characteristic polynomial $p$ of the matrix~$S_n$ satisfies the inequalities
\begin{equation}\label{Theorem.my.old.ineq.1}
\Delta_{k+2+4i}(p)<0,\qquad i=0,1,\ldots,\left[\frac{n-k-2}4\right],
\end{equation}
while all other Hurwitz determinants of $p$ are positive. The statement of the theorem now follows
from these inequalities and from Theorems~\ref{Theorem.general.Hurwitz.criterion} and~\ref{Theorem.almost.general.Hurwitz.criterion}.
\end{proof}

Converse theorem can also be established provided the given polynomial to satisfy the
inequalities~(\ref{Theorem.my.old.ineq.1}) while its other Hurwitz determinants are positive.
\begin{theorem}\label{Theorem.my.old.converse}
Let $\lambda_1,\ldots,\lambda_n$ be a sequence of complex numbers such that the
polynomial $p(z)=\prod\limits_{i=1}^n(z-\lambda_i)$ is a generalized Hurwitz polynomial
of order $\varkappa$ of type~I such that
\begin{equation}\label{Theorem.my.old.converse.ineq}
\Delta_{n-2\varkappa+2+4i}(p)<0,\qquad i=0,1,\ldots,\left[\frac{\varkappa-1}2\right],
\end{equation}
and other $\Delta_j(p)$ are positive. Then there exists a unique Schwarz matrix $S_n$ of the form~(\ref{Schwarz.matrix.2}) with $a>0$ and $k=n-2\varkappa$ such that $\sigma(S_n)=\{\lambda_1,\ldots,\lambda_n\}$.
\end{theorem}
\begin{proof}
Indeed, by the conditions of the theorem, all the Hurwitz determinants of the polynomial $p$ are nonzero, so according to Theorem~\ref{Theorem.Schwarz.matrix.Wall}, there exists a Schwarz matrix of the form~(\ref{main.matrix}) whose spectrum is $\{\lambda_1,\ldots,\lambda_n\}$. But from the formul\ae~(\ref{Formulae.1}), from the inequalities~(\ref{Theorem.my.old.converse.ineq}) (see also~(\ref{Theorem.my.old.ineq.1})) and from the positivity of all other
Hurwitz determinants~of~$p$, it follows that the sign pattern of the matrix $S_n$ must be as in~(\ref{Schwarz.matrix.2}) with $k=n-2\varkappa$.
\end{proof}

Analogously, one can prove the following theorem.
\begin{theorem}\label{Theorem.my.old.converse.2}
Let $\lambda_1,\ldots,\lambda_n$ be a sequence of complex numbers such that the
polynomial $p(z)=\prod\limits_{i=1}^n(z-\lambda_i)$ is almost generalized Hurwitz
of order $\varkappa$ of type~I such that
$$%\begin{equation}\label{Theorem.my.old.converse.ineq.2}
\Delta_{n-2\varkappa+1+4i}(p)<0,\qquad i=0,1,\ldots,\left[\frac{\varkappa-1}2\right],
$$%\end{equation}
and other $\Delta_j(p)$ are positive. Then there exists a unique Schwarz matrix $S_n$ of the form~(\ref{Schwarz.matrix.2}) with $a>0$ and $k=n-2\varkappa-1$ such that $\sigma(S_n)=\{\lambda_1,\ldots,\lambda_n\}$.
\end{theorem}
\begin{remark}
\emph{Note that the results due to H.~Wall and O.~Holtz (Theorems~\ref{Theorem.Routh.form} and~\ref{Theorem.Holtz}) follow from Theorems~\ref{Theorem.my.old},~\ref{Theorem.my.old.converse}
and~\ref{Theorem.my.old.converse.2} for $k=0$ and for $k=n-1$ (see Remark~\ref{remark.almost.gen}).}
\end{remark}

Finally, we show how to apply Theorem~\ref{Theorem.my.old} to a more particular case. Consider the following matrix studied in~\cite{Bebiano'11}
\begin{equation}\label{Bebiano.matrix}
J_n=
\left[\begin{array}{cccccc}
    a & 1 &  0 &\dots&   0   & 0 \\
    -c_1 & 0 &1 &\dots&   0   & 0 \\
     0  &c_2 & 0 &\dots&   0   & 0 \\
    \vdots&\vdots&\vdots&\ddots&\vdots&\vdots\\
     0  &  0  &  0  &\dots&0& 1\\
     0  &  0  &  0  &\dots&c_{n-1}&0\\
\end{array}\right], \quad a>0, c_j>0.
\end{equation}

By Theorem~\ref{Theorem.my.old}, the characteristic polynomial of this matrix is (almost)
generalized Hurwitz polynomial of order $\varkappa=\left[\frac{n-1}2\right]$ of type~II.
In~\cite{Bebiano'11}, there was posed the problem to find the condition on a sequence of complex
number to be the spectrum of the matrix~(\ref{Bebiano.matrix}). The following theorem solves
the direct and inverse problems for matrices of the form~(\ref{Bebiano.matrix}).
\begin{theorem}\label{Theorem.Bebiano.matrix}
The eigenvalues $\lambda_j$ of the matrix~(\ref{Bebiano.matrix}) are distributed in one of the following ways:
\begin{itemize}
\item[1)] $0<-\lambda_1<\lambda_2<-\lambda_3<\cdots<(-1)^n\lambda_{n-2}$, $\lambda_{n-1}=\overline{\lambda}_{n}\in\mathbb{C}$, and $\mathrm{Re}\lambda_{n}>0$;
\item[2)] $0<\lambda_1\leqslant\lambda_2<-\lambda_3<\lambda_4<-\lambda_5<\cdots<(-1)^n\lambda_{n}$;
\item[3)] for some $k$, $k=1,\ldots,\left[\frac{n-3}2\right]$,\\
$0<-\lambda_1<\lambda_2<\cdots<-\lambda_{2k-1}<\lambda_{2k}\leqslant\lambda_{2k+1}\leqslant\lambda_{2k+2}<\lambda_{2k+3}<\cdots
<(-1)^{n-1}\lambda_{n-1}<(-1)^n\lambda_{n}$;
\item[4)] $0<-\lambda_1<\lambda_2<-\lambda_3<\cdots<(-1)^n\lambda_{n-2}\leqslant\lambda_{n-1}\leqslant\lambda_{n}$,
and $(-1)^n\lambda_{n-2}<\lambda_{n-1}$ if $n=2l+1$.
\end{itemize}

Conversely, let $\lambda_1,\ldots,\lambda_n$ be a sequence of complex numbers satisfying one of the four conditions above, and $\sum\limits_{i=1}^{n}\lambda_i>0$. Then there exists a unique matrix $J_n$ of the form~(\ref{Bebiano.matrix}) such that $\sigma(J_n)=\{\lambda_1,\ldots,\lambda_n\}$.
\end{theorem}
\begin{proof}
As we already mentioned, by Theorem~\ref{Theorem.my.old}, the characteristic polynomial $p$ of the matrix~(\ref{Bebiano.matrix}) is generalized Hurwitz of order $\varkappa=\left[\frac{n-1}2\right]$ of type~II (if $n=2l+1$) or almost generalized Hurwitz of order $\varkappa=\left[\frac{n-1}2\right]$ of type~II (if $n=2l$). According to Definitions~\ref{def.almost.general.Hurw.poly},~\ref{def.general.Hurw.poly}
and~\ref{def.general.Hurw.poly.type.II}, the eigenvalues of the matrix~$J_n$ are distributed
in one of the four ways described in the statement of the theorem. Additionally, from the form
of the matrix~(\ref{Bebiano.matrix}) it follows that $\sum\limits_{i=1}^n\lambda_i=a>0$.

Conversely, let $\lambda_1,\ldots,\lambda_n$ be a sequence of complex numbers satisfying one of the four conditions above, and $\sum\limits_{i=1}^{n}\lambda_i>0$. Then the polynomial $p(z)=\prod\limits_{i=1}^n(z-\lambda_i)$ is generalized Hurwitz of order $\varkappa=\left[\frac{n-1}2\right]$ of type~II (if $n=2l+1$) or almost generalized Hurwitz of order $\varkappa=\left[\frac{n-1}2\right]$ of type~II (if $n=2l$) by Definitions~\ref{def.almost.general.Hurw.poly},~\ref{def.general.Hurw.poly}
and~\ref{def.general.Hurw.poly.type.II}. It is left to prove that $p$ satisfies
the inequalities~(\ref{Hurw.det.ineq}).

Let $n=2l+1$. Since $q(z):=p(-z)$ is generalized Hurwitz of type~I of order
$\varkappa=\left[\frac{n-1}2\right]=l$ by Definition~\ref{def.general.Hurw.poly.type.II}, we have
\begin{equation}\label{Theorem.Bebiano.matrix.proof.1}
\Delta_2(q)>0,\ \Delta_4(q)>0,\ \ldots,\Delta_{2l}(q)>0.
\end{equation}
and
$$%\begin{equation}\label{Theorem.Bebiano.matrix.proof.2}
l-1=V^{F}(1,\Delta_1(q),\Delta_3(q),\ldots,\Delta_{2l+1}(q)).
$$%\end{equation}
But $\Delta_1(q)=-\sum\limits_{i=1}^{n}(-\lambda_i)>0$, so $V^{F}(1,\Delta_1(q))=0$ and therefore we have
$$%\begin{equation}\label{Theorem.Bebiano.matrix.proof.2}
l-1=V^{F}(\Delta_1(q),\Delta_2(q),\ldots,\Delta_{2l+1}(q)).
$$%\end{equation}
Now the Frobenius rule of sign (see comments to Theorem~\ref{Theorem.general.Hurwitz.criterion} on p.~\pageref{Frobenius}) requires all the determinants $\Delta_3(q)$, $\Delta_5(q)$, \ldots,
$\Delta_{2l+1}(q)$ to be nonzero and satisfying the inequalities
\begin{equation}\label{Theorem.Bebiano.matrix.proof.2}
(-1)^{i-1}\Delta_{2i-1}(q)>0,\qquad i=1,\ldots,l+1.
\end{equation}
From the inequalities~(\ref{Theorem.Bebiano.matrix.proof.1})--(\ref{Theorem.Bebiano.matrix.proof.2}),
Theorem~\ref{Theorem.gen.Hurw.Schwarz.odd}, and the formul\ae~(\ref{Formulae.1}) we obtain that there exists a unique matrix of the form
$$%\begin{equation}\label{Bebiano.matrix.2}
\left[\begin{array}{cccccc}
    -a & 1 &  0 &\dots&   0   & 0 \\
    -c_1 & 0 &1 &\dots&   0   & 0 \\
     0  &c_2 & 0 &\dots&   0   & 0 \\
    \vdots&\vdots&\vdots&\ddots&\vdots&\vdots\\
     0  &  0  &  0  &\dots&0& 1\\
     0  &  0  &  0  &\dots&c_{n-1}&0\\
\end{array}\right], \quad a>0, c_j>0.
$$%\end{equation}
whose characteristic polynomial is $q$. Now Remark~\ref{remark.type.II} gives us the assertion of the theorem for $n=2l+1$. The case $n=2l$ can be established analogously.
\end{proof}

Note that the additional condition $\sum\limits_{i=1}^{n}\lambda_i>0$ is substantial for solution of the inverse problem for the matrix~(\ref{Bebiano.matrix}). If this number is negative, then the matrix must have another sign pattern. But if this number is zero, the inverse problem has no solution.

We finish by noting that using results of Section~\ref{section:spectral.problems} one can find more examples of sign patterns of Schwarz matrices with (almost) generalized Hurwitz characteristic polynomials. At least, given a Schwarz matrix, one can always say if its characteristic polynomial is (almost) generalized Hurwitz or not.

\bigskip
{\bf Acknowledgment.} The author thanks N.~Bebiano and C.~da Fonseca for helpful discussions.

%%%%%%%%%%%%%%%%%%%%%%%%%%%%%%%%%%%%%%%%%%%%%%%%%%%%%%%%%%%%%

\end{document}